\documentclass{amsart}
\usepackage{amssymb}
\usepackage{amsfonts}

\newtheorem{theorem}{Theorem}
\theoremstyle{plain}

\newtheorem{definition}{Definition}

\newtheorem{remark}{Remark}

\numberwithin{equation}{section}
\input{tcilatex}

\begin{document}
\title[Moduli Spaces of $J$-holomorphic Curves with General Jet Constraints]{%
Moduli Spaces of $J$-holomorphic Curves with General Jet Constraints}
\author{Ke Zhu}
\address{Department of Mathematics\\
The Chinese University of Hong Kong\\
Shatin, Hong Kong}
\email{kzhu@math.cuhk.edu.hk}
\thanks{}
\date{November 9, 2009}

\begin{abstract}
In this paper, we prove that the tagent map of the holomorphic $k$-jet
evaluation $j_{hol}^{k}$ from the mapping space to holomorphic $k$-jet
bundle, when restricted on the universal moduli space $\mathcal{M}_{1}^{\ast
}\left( \Sigma ,M,\beta \right) $ of simple $J$-holomorphic curves with one
marked point, is surjective. From this we derive that for generic $J$, the
moduli space of simple $J$-holomorphic curves with general jet constraints
at marked points is a smooth manifold of expected dimension.
\end{abstract}

\maketitle

\section{ Introduction}

\bigskip

\bigskip Let $\left( M,\omega \right) $ be a symplectic manifold of
dimension $2n$. Denote by $\mathcal{J}_{\omega }$ the set of almost complex
structures $J$ on $M$ compatible with $\omega $. Let $\Sigma $ be a compact
oriented surface without boundary, and $\left( j,u\right) $ a pair of
complex structure $j$ on $\Sigma $ and a map $u:\Sigma \rightarrow M$. We
say $\left( j,u\right) $ is a $J$-holomorphic curve if $\overline{\partial }%
_{J,j}u:=\frac{1}{2}\left( du+J\circ du\circ j\right) =0$. We let $\mathcal{M%
}_{1}\left( \Sigma ,M,J;\beta \right) $ be the standard moduli space of $J$%
-holomorphic curves in class $\beta \in H_{2}\left( M,\mathbb{Z}\right) $
with one marked point, and $\mathcal{M}_{1}^{\ast }\left( \Sigma ,M,J;\beta
\right) $ be the set of simple (i.e. somewhere injective) $J$-holomorphic
curves in $\mathcal{M}_{1}\left( \Sigma ,M;\beta \right) $.

Since the birth of the theory of $J$-holomorphic curves, moduli spaces of $J$%
-holomorphic curves with constraints at marked points have lead to finer
symplectic invarints like Gromov-Witten invariants and quantumn cohomology.  
$J$-holomorphic curves with embedding property also plays important role in
low dimesional symplectic geometry, like the works of \cite{HT} and \cite%
{Wen}. These constraints all can be viewed as partial differential relations
in the $0$-jet and $1$-jet bundles. In relative Gromov-Witten theory,
contact order of $J$-holomorphic curves with given symplectic hypersurfaces
(divisors) was used to define the relevant moduli spaces.  In the work of
Cieliebak-Mohke \cite{CM} and Oh \cite{Oh}, the authors studied the moduli
space of $J$-holomorphic curves with prescribed vanishing orders of
derivatives at marked points. All these are vanishing conditions in $k$-jets
bundles. It is then natural to ask what properties we can expect for moduli
spaces of $J$-holomorphic curves with general constraints in jet bundles
(while all constraints in previous examples are\emph{\ zero sections} in
various jet bundles).

The main purpose of this paper is to confirm that for a wide class of closed
partial differential relations in \emph{holomorphic jet bundles} (Definition %
\ref{hol-jet-bdl}, orginially defined in \cite{Oh}), the moduli spaces of $J$%
-holomorphic curves from $\Sigma $ to $M$ with given constraints at marked
points behave well for generic $J$ (Theorem \ref{S}). Namely, they are
smooth manifolds of dimension predicted by index theorem, and all elements
in the moduli spaces are Fredholm regular. During the proof it appears that
holomorphic jet bundles are the natural framework to put jet constraints for 
$J$-holomorphic curves in order to obtain regularity of their moduli spaces.
The regularity of $J$-holomorphic curve moduli spaces \emph{fails} for
general constraints in \emph{usual} jet bundle (Remark \ref{fail-usual-jet}%
), but still holds in a special case when the moduli space consists of
immersed $J$-holomorphic curves (Theorem \ref{usual-jet}).

The key of the proof is to establish the sujective property of the
linearization of $k$-jet evaluations on the universal moduli spaces of $J$%
-holomorphic curves at marked points insider the mapping space, including
the parameter $J\in \mathcal{J}_{\omega }$(Theorem \ref{Submersion}). It is
important to take the evaluations in \emph{holomorphic} \emph{jet bundles}
in order to get the surjectivity of the linearization of the $k$-jet
evaluation map.

Since $\mathcal{J}_{\omega }$ is a huge parameter space to deform $J$%
-holomorphic curves, the sujective property here is a reminiscence of the
classic Thom transversality theorem, which says that the $k$-jet evaluation
on smooth mapping space to the $k$-jet bundle is transversal to any section
there.

The framework of the paper is similar to \cite{Oh}, which in turn is a
higher jet generalization of \cite{OZ} for $1$-jet transversality of $J$%
-holomorphic curves. The main steps of the paper are in order:

\begin{enumerate}
\item We set up the Banach bundle including the finite dimensional
holomorphic $k$-jet subbundle $J_{hol}^{k}\left( \Sigma ,M\right) $ over the
mapping space $\mathcal{F}_{1}\left( \Sigma ,M\right) \times \mathcal{J}%
_{\omega }$ and define the section $\Upsilon _{k}=\left( \overline{\partial }%
,j_{hol}^{k}\right) $, where%
\begin{equation*}
\Upsilon _{k}:\left( \left( u,j,z_{0}\right) ,J\right) \rightarrow \left( 
\overline{\partial }_{j,J}u,j_{hol}^{k}\left( u\left( z_{0}\right) \right)
\right) .
\end{equation*}%
We inteprete the universal $J$-holomorphic curve moduli space as 
\begin{equation*}
\mathcal{M}\left( \Sigma ,M\right) =\overline{\partial }^{-1}\left( 0\right)
=\Upsilon _{k}^{-1}\left( 0,J_{hol}^{k}\left( \Sigma ,M\right) \right) .
\end{equation*}

\item We compute the linearization $D\Upsilon _{k}$ of the section $\Upsilon
_{k}$. We express the submersion property of $\Upsilon _{k}$ as the
solvability of a system of equations $D\Upsilon _{k}\left( \xi ,B\right)
=\left( \gamma ,\alpha \right) $ for any $\left( \gamma ,\alpha \right) $,
where $\left( \xi ,B\right) \in T_{u}\mathcal{F}_{1}\left( \Sigma ,M\right)
\times T_{J}\mathcal{J}_{\omega }$, or equivalently, the vanishing of the
cokernal element $\left( \eta ,\zeta \right) $ in the Fredholm alternative
system: $F\left\langle \left( \xi ,B\right) ,\left( \eta ,\zeta \right)
\right\rangle =0$ for all $\left( \xi ,B\right) $. This is called the
cokernal equation.

\item Using the abundance of $B\in T_{J}\mathcal{J}_{\omega }$ we get supp$%
\eta \subset \left\{ z_{0}\right\} $. Then we use a structure theorem in
distribution to write $\eta $ as a linear combination of $\delta $ function
and its derivatives at $z_{0}$, up to $\left( k-1\right) $-th order
derivatives.

\item Since supp$\eta \subset \left\{ z_{0}\right\} $ the cokernal equation
is supported at $z_{0}.$ We replace the $\xi $ in the cokernal equation by $%
\xi +h$ where $h=h\left( z,\overline{z}\right) $ is a suitable polynomial in
local coordinates nearby $z_{0}$, and set $B=0$, so that the cokernal
equation is reduced to $\left\langle D_{u}\overline{\partial }_{j,J}\xi
,\eta \right\rangle =0$ for all $\xi $. The crucial observation is that to
get $\left\langle D_{u}\overline{\partial }_{j,J}\xi ,\eta \right\rangle =0$
we do not need so strong conditions of vanishing of $1\thicksim k$%
-derivatives of $u$ at $z_{0}$ as in \cite{Oh} and \cite{CM}. This is by
exploring the flexibility of $h$ to get rid of redundant terms from the
original cokernal equation.

\item Then we apply elliptic regularity to conclude $\eta =0$ and
consequently $\zeta =0$. Therefore we get the sujectivity  of $D\Upsilon _{k}
$ and $Dj_{hol}^{k}$.

\item Finally, there is an obstruction in step 4 to get $h$ when $\zeta
_{k}=0$, where $\zeta _{k}$ is the $k$-th component of $\zeta $. But when $%
\zeta _{k}=0$ the cokernal equation is reduced to the $\left( k-1\right) $%
-jet evaluation setting, so we still get $\left( \eta ,\zeta \right) =\left(
0,0\right) $ by induction on $k$.
\end{enumerate}

\bigskip \textbf{Acaknowledgement. }The author would like to thank Yakov
Eliashberg to suggest the generalization from  \cite{Oh} to general PDE
relations. He would also like to thank Yong-Geun Oh on past discussions in
holomorphic jet transversality.

\section{\protect\bigskip Holomorphic jet bundle}

We recall the holomorphic jet bundle from \cite{Oh}. Given $\Sigma ,M$, and $%
\left( z,x\right) \in \Sigma \times M$, the $k$-jet with source $z$ and
target $x$ is defined as (see \cite{Hir})%
\begin{equation*}
J_{z,x}^{k}\left( \Sigma ,M\right) =\dprod\limits_{l=0}^{k}Sym^{l}\left(
T_{z}\Sigma ,T_{x}M\right) ,
\end{equation*}%
where $Sym^{l}\left( T_{z}\Sigma ,T_{x}M\right) $ is the set of $l$%
-multilinear maps from $T_{z}\Sigma $ to $T_{x}M$ for $l\geq 1$. Here for
convenience we have set $Sym^{0}\left( T_{z}\Sigma ,T_{x}M\right) =M$. Let 
\begin{equation*}
J^{k}\left( \Sigma ,M\right) =\dbigcup\limits_{\left( z,x\right) \in \Sigma
\times M}J_{z,x}^{k}\left( \Sigma ,M\right) 
\end{equation*}%
be the $k$-jet bundle over $\Sigma \times M$. For the mapping space 
\begin{equation*}
\mathcal{F}_{1}\left( \Sigma ,M;\beta \right) =\left\{ \left( \left( \Sigma
,j\right) ,u\right) |j\in \mathcal{M}\left( \Sigma \right) ,z\in \Sigma
,u:\Sigma \rightarrow M,\left[ u\right] =\beta \right\} ,
\end{equation*}%
we consider the map%
\begin{equation*}
\mathcal{F}_{1}\left( \Sigma ,M;\beta \right) \rightarrow \Sigma \times M,%
\text{ }\left( u,j,z\right) \rightarrow \left( z,u\left( z\right) \right) .
\end{equation*}%
By this map we can pull back the bundle $J^{k}\left( \Sigma ,M\right)
\rightarrow \Sigma \times M$ to the base $\mathcal{F}_{1}\left( \Sigma
,M;\beta \right) $. By abusing notation, we still call the resulted bundle
by $J^{k}\left( \Sigma ,M\right) .$Then $J^{k}\left( \Sigma ,M\right)
\rightarrow \mathcal{F}_{1}\left( \Sigma ,M;\beta \right) $ is a finite
dimensional vector bundle over the Banach manifold $\mathcal{F}_{1}\left(
\Sigma ,M;\beta \right) $. We define the $k$-jet evaluation 
\begin{equation*}
j^{k}:\mathcal{F}_{1}\left( \Sigma ,M;\beta \right) \rightarrow J^{k}\left(
\Sigma ,M\right) ,\text{ }j^{k}\left( \left( u,j\right) ,z\right)
=j_{z}^{k}u\in J_{z,u\left( z\right) }^{k}\left( \Sigma ,M\right) .
\end{equation*}%
Then $j^{k}$ is a smooth section. Classic Thom transversality theorem says
that $j^{k}$ is transversal to any section in $J^{k}\left( \Sigma ,M\right) $%
.

Now we turn to the case when $\Sigma $ and $M$ are equipped with (almost)
complex structures $j$ and $J$ respectively. The corresponding concept is
the holomorphic jet bundle defined in \cite{Oh}. With respect to $\left(
j_{z},J_{x}\right) $, $Sym_{z,x}^{l}\left( \Sigma ,M\right) $ splits into
summands indexed by the bigrading $\left( p,q\right) $ for $p+q=k$: 
\begin{equation*}
Sym_{z,x}^{l}\left( \Sigma ,M\right) =Sym^{\left( l,0\right) }\left(
T_{z}\Sigma ,T_{x}M\right) \oplus Sym^{\left( 0,l\right) }\left( T_{z}\Sigma
,T_{x}M\right) \oplus \text{\textquotedblleft mixed parts\textquotedblright }
\end{equation*}%
Let 
\begin{eqnarray*}
H_{j_{z,}J_{x}}^{\left( l,0\right) }\left( \Sigma ,M\right) &=&Sym^{\left(
l,0\right) }\left( T_{z}\Sigma ,T_{x}M\right) , \\
H_{j,J}^{\left( l,0\right) }\left( \Sigma ,M\right)
&=&\dbigcup\limits_{\left( z,x\right) \in \Sigma \times
M}H_{j_{z,}J_{x}}^{\left( l,0\right) }\left( \Sigma ,M\right) .
\end{eqnarray*}%
Given $\left( j,J\right) $, the $\left( j,J\right) $\emph{-holomorphic jet
bundle} $J_{\left( j,J\right) hol}^{k}\left( \Sigma ,M\right) $ is defined
as 
\begin{equation}
J_{\left( j,J\right) hol}^{k}\left( \Sigma ,M\right)
=\dprod\limits_{l=0}^{k}H_{j,J}^{\left( l,0\right) }\left( \Sigma ,M\right) ,
\label{HolSum}
\end{equation}%
which is a finite dimensional vector bundle over $\Sigma \times M$.

We define the bundle%
\begin{equation*}
J_{hol}^{k}\left( \Sigma ,M\right) =\dbigcup\limits_{\left( j,J\right) \in 
\mathcal{M}\left( \Sigma \right) \times \mathcal{J}_{\omega }}J_{\left(
j,J\right) hol}^{k}\left( \Sigma ,M\right) .
\end{equation*}%
$J_{hol}^{k}\left( \Sigma ,M\right) \rightarrow \Sigma \times M\times 
\mathcal{M}\left( \Sigma \right) \times \mathcal{J}_{\omega }$ is a finite
dimensional vector bundle over the base Banach manifold. Using the pull back
of the map%
\begin{equation*}
ev:\mathcal{F}_{1}\left( \Sigma ,M;\beta \right) \times \mathcal{J}_{\omega
}\rightarrow \Sigma \times M\times \mathcal{M}\left( \Sigma \right) \times 
\mathcal{J}_{\omega },\text{ }\left( \left( u,j\right) ,z,J\right)
\rightarrow \left( z,u\left( z\right) ,j,J\right) ,
\end{equation*}%
$ev^{\ast }\left( J_{hol}^{k}\left( \Sigma ,M\right) \right) $ is a finite
dimensional vector bundle over the Banach manifold $\mathcal{F}_{1}\left(
\Sigma ,M;\beta \right) \times \mathcal{J}_{\omega }$. By abusing of
notation, we still call $ev^{\ast }\left( J_{hol}^{k}\left( \Sigma ,M\right)
\right) $ by $J_{hol}^{k}\left( \Sigma ,M\right) $.

\begin{definition}
\label{hol-jet-bdl} $J_{hol}^{k}\left( \Sigma ,M\right) \rightarrow \mathcal{%
F}_{1}\left( \Sigma ,M;\beta \right) \times \mathcal{J}_{\omega }$ is called
the holomorphic $k$-jet bundle.
\end{definition}

Let $\pi _{hol}:J^{k}\left( \Sigma ,M\right) \rightarrow J_{hol}^{k}\left(
\Sigma ,M\right) $ be the bundle projection. We define the holomorphic $k$%
-jet evaluation%
\begin{equation*}
j_{hol}^{k}=\pi ^{hol}\circ j^{k}.
\end{equation*}%
It is not hard to see $j_{hol}^{k}$ is a smooth section of the Banach bundle 
$\mathcal{F}_{1}\left( \Sigma ,M;\beta \right) \times \mathcal{J}_{\omega
}\rightarrow J_{hol}^{k}\left( \Sigma ,M\right) $. According to the summand $%
\left( \ref{HolSum}\right) $, we write $j_{hol}^{k}$ in components 
\begin{equation*}
j_{hol}^{k}=\dprod\limits_{l=0}^{k}\sigma ^{l},
\end{equation*}%
where the $l$-th component is%
\begin{equation*}
\sigma ^{l}:\mathcal{F}_{1}\left( \Sigma ,M;\beta \right) \times \mathcal{J}%
_{\omega }\rightarrow H_{j,J}^{\left( l,0\right) }\left( \Sigma ,M\right) ,%
\text{ }\left( \left( u,j\right) ,z,J\right) \rightarrow \pi
_{j,J}^{hol}\left( d^{l}u\left( z\right) \right) .
\end{equation*}%
We remark that if $J$ is integrable, $\sigma ^{l}$ corresponds to the $l$-th
holomorphic derivative $\frac{\partial ^{l}}{\partial z^{l}}u$ of $u$ at $z$.

The important point is that the holomorphic $k$-jet bundle and the section $%
j_{hol}^{k}$ are canonically associated to the pair $\left( \Sigma ,j\right) 
$ and $\left( M,J\right) $ in the \textquotedblleft off-shell
level\textquotedblright , i.e. on the space of all smooth maps, not only $J$%
-holomorphic maps. This enables us to formulate the jet constraints for $J$%
-holomorphic maps as some submanifold in the bundle $J_{hol}^{k}\left(
\Sigma ,M\right) \rightarrow \mathcal{F}_{1}\left( \Sigma ,M;\beta \right)
\times \mathcal{J}_{\omega }$.

\section{Fredholm set up}

The Fredholm set up is the same as in \cite{Oh}, with the simplification
that we only need one marked point on $\Sigma $. The case with more marked
points has no essential difference. We introduce the standard bundle 
\begin{equation*}
\mathcal{H}^{^{\prime \prime }}=\dbigcup\limits_{\left( \left( u,j\right)
,J\right) }\mathcal{H}_{\left( \left( u,j\right) ,J\right) }^{^{\prime
\prime }},\text{ \ \ }\mathcal{H}_{\left( \left( u,j\right) ,J\right)
}^{^{\prime \prime }}=\Omega _{j,J}^{\left( 0,1\right) }\left( u^{\ast
}TM\right) 
\end{equation*}%
and define the section 
\begin{equation*}
\Upsilon _{k}:\mathcal{F}_{1}\left( \Sigma ,M;\beta \right) \times \mathcal{J%
}_{\omega }\rightarrow \mathcal{H}^{^{\prime \prime }}\times J^{k}\left(
\Sigma ,M\right) 
\end{equation*}%
as 
\begin{equation*}
\Upsilon _{k}\left( \left( u,j\right) ,z,J\right) =\left( \overline{\partial 
}\left( u,j,J\right) ;j_{hol}^{k}\left( u,j,J,z\right) \right) ,
\end{equation*}%
where 
\begin{equation*}
\overline{\partial }\left( u,j,J\right) =\overline{\partial }_{j,J}\left(
u\right) =\frac{du+J\circ du\circ j}{2}.
\end{equation*}%
Given $\beta \in H_{2}\left( M,\mathbb{Z}\right) $, let 
\begin{equation*}
\mathcal{M}_{1}\left( \Sigma ,M;\beta \right) =\dbigcup\limits_{J\in 
\mathcal{J}_{\omega }}\mathcal{M}_{1}\left( \Sigma ,M,J;\beta \right) 
\end{equation*}%
be the universal moduli space of $J$-holomorphic curves in class $\beta $
with one marked point. Its open subset consisting of somewhere injective $J$%
-holomorphic curves is denoted by $\mathcal{M}_{1}^{\ast }\left( \Sigma
,M;\beta \right) $. It is a standard fact in symplectic geometry that $%
\mathcal{M}_{1}^{\ast }\left( \Sigma ,M;\beta \right) $ is a Banach manifold.

Now we make precise the necessary regularity requirement for the Banach
manifold set-up:

\begin{enumerate}
\item To make sense of the evaluation of $j^{k}u$ at a point $z$ on $\Sigma $%
, we need to take at least $W^{k+1,p}$-completion with $p>2$ of $\mathcal{F}%
_{1}\left( \Sigma ,M;\beta \right) $ so $j^{k}u\in W^{1,p}\hookrightarrow
C^{0}$. To make the section $\Upsilon _{k}$ differentiable we need to take $%
W^{k+2,p}$ completion, since in $\left( \ref{Lin-j}\right) $ $\left(
k+1\right) $-th derivatives of $u$ are involved. To apply Sard-Smale
theorem, we actually need to take $W^{N,p}$ completion with sufficiently
large $N=N\left( \beta ,k\right) $.

\item We provide $\mathcal{H}^{\prime \prime }$ with topology of a $W^{N,p}$
Banach bundle.

\item We also need to provide the Banach manifold structure of $\mathcal{J}%
_{\omega }$. We can borrow Floer's scheme \cite[F]{F} for this whose details
we refer readers thereto.
\end{enumerate}

\section{Transversality}

\begin{theorem}
\label{Submersion}At every $J$-holomorphic curve $\left( \left( u,j\right)
,z,J\right) \in $ $\mathcal{M}_{1}^{\ast }\left( \Sigma ,M;\beta \right)
\subset \mathcal{F}_{1}\left( \Sigma ,M;\beta \right) \times \mathcal{J}%
_{\omega }$, the linearization $D\Upsilon _{k}$ of the map 
\begin{equation*}
\Upsilon _{k}=\left( \overline{\partial },\text{ }j_{hol}^{k}\right) :%
\mathcal{F}_{1}\left( \Sigma ,M;\beta \right) \times \mathcal{J}_{\omega
}\rightarrow \mathcal{H}^{^{\prime \prime }}\times J_{hol}^{k}\left( \Sigma
,M\right)
\end{equation*}%
is surjective. Especially the linearization $Dj_{hol}^{k}$ of the
holomorphic $k$-jet evaluation 
\begin{equation*}
j_{hol}^{k}:\mathcal{F}_{1}\left( \Sigma ,M;\beta \right) \times \mathcal{J}%
_{\omega }\rightarrow J_{hol}^{k}\left( \Sigma ,M\right)
\end{equation*}%
on $\mathcal{M}_{1}^{\ast }\left( \Sigma ,M;\beta \right) $\ is sujective.
\end{theorem}

To prove theorem we need to verify that at each $\left( \left( u,j\right)
,z,J\right) \in $ $\mathcal{M}_{1}^{\ast }\left( \Sigma ,M;\beta \right) $ ,
the system of equations%
\begin{eqnarray}
D_{J,\left( j,u\right) }\overline{\partial }\left( B,\left( b,\xi \right)
\right) &=&\gamma  \label{Lin-dbar} \\
D_{J,\left( j,u\right) }j_{hol}^{k}\left( B,\left( b,\xi \right) \right)
\left( z\right) +\nabla _{v}\left( j_{hol}^{k}\left( u\right) \right) \left(
z\right) &=&\alpha  \label{Lin-j}
\end{eqnarray}%
has a solution $\left( B,\left( b,\xi \right) ,v\right) \in T_{J}\mathcal{J}%
_{\omega }\times T_{j}\mathcal{M}\left( \Sigma \right) \times T_{u}\mathcal{F%
}_{1}\left( \Sigma ,M;\beta \right) \times T_{z}\Sigma $ \ for each given
data 
\begin{equation*}
\gamma \in \Omega _{N-1,p}^{\left( 0,1\right) }\left( u^{\ast }TM\right) ,%
\text{ \ \ }\zeta =\left( \text{\ }\zeta _{0,}\text{\ }\zeta _{1,}\ldots 
\text{\ }\zeta _{k}\right) \in J_{hol}^{k}\left( T_{z}\Sigma ,T_{u\left(
z\right) }M\right) .
\end{equation*}%
It will be enough to consider the triple with $b=0$ and $v=0$ which we will
assume from now on.

We compute the $D_{J,\left( j,u\right) }j_{hol}^{k}\left( B,\left( b,\xi
\right) \right) \left( z\right) $. \ It is enough to compute $D_{J,\left(
j,u\right) }\sigma ^{l}\left( B,\left( 0,\xi \right) \right) \left( z\right) 
$ for $l=0,1,\cdots k$. We have 
\begin{equation}
D_{J,\left( j,u\right) }\sigma ^{l}\left( B,\left( 0,\xi \right) \right)
\left( z\right) =\pi _{hol}\left( \left( \nabla _{du}\right) ^{l}\xi \left(
z\right) \right) +\Sigma _{0\leq s,t\leq l}B\left( z\right) \cdot
F_{st}\left( z\right) \left( \left( \nabla _{du}\right) ^{s}\xi \left(
z\right) ,\nabla ^{t}u\left( z\right) \right)  \label{d-jet-local}
\end{equation}%
where $F_{st}\left( z\right) \left( \cdot ,\cdot \right) $ is some\
vector-valued monomial, and $B\left( z\right) $ is a matrix valued function,
both smoothly depending on $z$. There is no derivative of $B$ in the above
formula, because for any $l$, $\sigma ^{l}$ is the projection of the tensor $%
d^{l}u\in Sym^{l}\left( T_{z}\Sigma ,T_{u\left( x\right) }M\right) $ to $%
Sym_{j,J}^{\left( l,0\right) }\left( T_{z}\Sigma ,T_{u\left( x\right)
}M\right) $, and the projection only involves $J$ but not its derivatives.
Since $u$ is $\left( j\text{,}J\right) $-holomorphic, it also follows that 
\begin{equation}
\pi _{hol}\left( \left( \nabla _{du}\right) ^{l}\xi \left( z\right) \right)
=\left( \nabla _{du}^{^{\prime }}\right) ^{l}\xi \left( z\right) ,
\label{power}
\end{equation}%
where $\nabla _{du}^{^{\prime }}=\pi _{hol}\nabla _{du}=D_{u}\partial _{j,J}$%
. There is a formula for $D_{u}\partial _{j,J}$ and $D_{u}\overline{\partial 
}_{j,J}$ nearby $z_{0}$ (see \cite{Si}): 
\begin{eqnarray}
D_{u}\overline{\partial }_{j,J}\xi &=&\overline{\partial }\xi +A\left(
z\right) \partial \xi +C\left( z\right) \xi  \label{Sikarov} \\
D_{u}\partial _{j,J}\xi &=&\partial \xi +G\left( z\right) \overline{\partial 
}\xi +H\left( z\right) \xi  \notag
\end{eqnarray}%
where $A\left( z\right) ,C\left( z\right) ,G\left( z\right) ,H\left(
z\right) $ are matrix-valued smooth functions, all vanishing at $z_{0}$.

Now we study the solvability of $\left( \ref{Lin-dbar}\right) $ and $\left( %
\ref{Lin-j}\right) $ by Fredholm alternative. We regard 
\begin{equation*}
\Omega _{N-1,p}^{\left( 0,1\right) }\left( u^{\ast }TM\right) \times
J_{hol}^{k}\left( T_{z}\Sigma ,T_{u\left( z\right) }M\right)
\end{equation*}%
as a Banach space with the norm 
\begin{equation*}
\left\Vert \cdot \right\Vert _{N-1,p}+\Sigma _{l=1}^{k}\left\vert \cdot
\right\vert _{l}
\end{equation*}%
where $\left\vert \cdot \right\vert _{l}$ is any norm induced by an inner
product on the $2n$-dimensional vector space $Sym_{j,J}^{\left( l,0\right)
}\left( T_{z}\Sigma ,T_{u\left( z\right) }M\right) \simeq \mathbb{C}^{n}$.

We denote the natural pairing 
\begin{equation*}
\Omega _{N-1,p}^{\left( 0,1\right) }\left( u^{\ast }TM\right) \times \left(
\Omega _{N-1,p}^{\left( 0,1\right) }\left( u^{\ast }TM\right) \right) ^{\ast
}\rightarrow \mathbb{R}
\end{equation*}%
by $\left\langle \cdot ,\cdot \right\rangle $ and the inner product on $%
Sym_{j,J}^{\left( l,0\right) }\left( T_{z}\Sigma ,T_{u\left( z\right)
}M\right) $ by $\left( \cdot ,\cdot \right) _{z}$.

Let $\left( \eta ,\zeta \right) \in \left( \Omega _{N-1,p}^{\left(
0,1\right) }\left( u^{\ast }TM\right) \right) ^{\ast }\times
J_{hol}^{k}\left( T_{z}\Sigma ,T_{u\left( z\right) }M\right) $ for $\zeta
=\left( \zeta _{1,\cdots ,}\zeta _{k}\right) $ such that 
\begin{equation}
\left\langle D_{J,\left( j,u\right) }\overline{\partial }\left( B,\left(
0,\xi \right) \right) ,\eta \right\rangle +\Sigma _{l=1}^{k}\left(
D_{J,\left( j,u\right) }\sigma ^{l}\left( B,\left( 0,\xi \right) \right)
\left( z\right) ,\zeta _{l}\right) _{z}=0  \label{coker-equ}
\end{equation}%
for all $\xi \in \Omega _{N-1,p}^{\left( 0,1\right) }\left( u^{\ast
}TM\right) $ and $B\in T_{J}\mathcal{J}_{\omega }$. We want to show $\left(
\eta ,\zeta \right) =\left( 0,0\right) $. The idea is to change the above
equation into%
\begin{equation*}
\left\langle D_{J,\left( j,u\right) }\overline{\partial }\left( B,\left(
0,\xi \right) \right) ,\eta \right\rangle =0
\end{equation*}%
for all $\xi $ and $B$ by judiciously modifying $\xi $ by a Taylor
polynomial nearby $z$, and then use standard techniques in $J$-holomorphic
curve theory to show $\eta =0$, and after that use Cauchy integral to show $%
\zeta =0$. We first deal with $N=k$ case, and later raise the regularity by
ellipticity of Cauchy-Riemann equation.

\ \ Let $\xi =0$, then $\left( \ref{coker-equ}\right) $ becomes%
\begin{equation*}
\left\langle \frac{1}{2}B\circ du\circ j,\eta \right\rangle =0.
\end{equation*}%
Using the abundance of $B\in T_{J}\mathcal{J}_{\omega }$, and that $u$ is a
simple $J$-holomorphic curve, by standard technique (for example \cite{MS})
we get $\eta =0$ on $\Sigma \backslash \left\{ z_{0}\right\} $, namely supp$%
\eta \subset \left\{ z_{0}\right\} $. Since $\eta \in \left( W^{k,p}\right)
^{\ast }$, by the structure theorem of distribution with point support (see 
\cite{GS}), we have%
\begin{equation}
\eta =P\left( \frac{\partial }{\partial z},\frac{\partial }{\partial 
\overline{z}}\right) \delta _{z_{0}}  \label{point-distribution}
\end{equation}%
where $\delta _{z_{0}}$ is the delta function supported at $z_{0}$, and $P$
is a polynomial in two variables with degree $\leq k-1$: this is because the
evaluation at a point of the $k$-th derivative of $W^{k,p}$ maps does not
define a continuous functional on $W^{k,p}$.

\ \ Let $B=0$. By $\left( \ref{d-jet-local}\right) $ and $\left( \ref%
{point-distribution}\right) $, $\left( \ref{coker-equ}\right) $ becomes%
\begin{equation}
\left\langle D_{u}\overline{\partial }_{j,J}\xi ,\eta \right\rangle +\left(
\left( D_{J,\left( j,u\right) }j_{hol}^{k}\right) \xi \left( z_{0}\right)
,\zeta \right) _{z_{0}}=0.  \label{coker-equ-0}
\end{equation}%
Since $\xi $ is arbitrary, we can replace $\xi $ by $\xi +\chi \left(
z\right) h\left( z,\overline{z}\right) $ in the above identity, where $%
h\left( z,\overline{z}\right) $ is a vector-valued polynomial in $z$ and $%
\overline{z}$, and $\chi \left( z\right) $ is a smooth cut-off function
equal to $1$ in a coordinate neighborhood of $z_{0}$ and $0$ outside a
slightly larger neighborhood, so that $\chi \left( z\right) h\left( z,%
\overline{z}\right) $ is a well defined and smooth on whole $\Sigma $. We
want $\left( \ref{coker-equ-0}\right) $ becomes $\left\langle D_{u}\overline{%
\partial }_{j,J}\xi ,\eta \right\rangle =0$ after that replacement. For this
purpose the $h\left( z,\overline{z}\right) $ should satisfy 
\begin{equation}
\left\langle D_{u}\overline{\partial }_{j,J}h,P\left( \frac{\partial }{%
\partial z},\frac{\partial }{\partial \overline{z}}\right) \delta
_{z_{0}}\right\rangle +\left( \left( D_{J,\left( j,u\right)
}j_{hol}^{k}\right) h\left( z_{0}\right) ,\zeta \right) _{z_{0}}=-\left(
\left( D_{J,\left( j,u\right) }j_{hol}^{k}\right) \xi \left( z_{0,}\overline{%
z}_{0}\right) ,\zeta \right) _{z_{0}}  \label{modify-terms}
\end{equation}%
After simplification, the above is a differential equation about $h$:%
\begin{equation}
Q\left( \frac{\partial }{\partial z},\frac{\partial }{\partial \overline{z}}%
\right) h\left( z_{0,}\overline{z}_{0}\right) =w  \label{modify-equ}
\end{equation}%
where $Q\left( s,t\right) $ is a vector-valued polynomial in two variables $%
s,t$\thinspace , $Q\left( \frac{\partial }{\partial z},\frac{\partial }{%
\partial \overline{z}}\right) $ acts on $h\left( z,\overline{z}\right) $
with vector coefficients paired with those of $h$ by inner product, and $%
w:=-\left( \left( D_{J,\left( j,u\right) }j_{hol}^{k}\right) \xi \left(
z_{0}\right) ,\zeta \right) _{z_{0}}$ is a constant.

\emph{Here comes the crucial observation}: when $\zeta _{k}\neq 0,$ the
highest degree of $s$ in $Q\left( s,t\right) $ is in the term $\zeta
_{k}s^{k}$. This is because $P\left( s,t\right) $ has degree$\leq k-1$ and
after integration by parts, $\frac{\partial }{\partial z}$ can fall at $D_{u}%
\overline{\partial }_{j,J}h$ of most $\left( k-1\right) $ times, and in $%
\left( \ref{Sikarov}\right) $ 
\begin{equation*}
D_{u}\overline{\partial }_{j,J}h=\overline{\partial }h+A\left( z\right)
\partial h+C\left( z\right) h,
\end{equation*}%
where $A\left( z_{0}\right) =0$. On the other hand, in $\left( D_{J,\left(
j,u\right) }j_{hol}^{k}\right) h$, by $\left( \ref{d-jet-local}\right)
,\left( \ref{power}\right) ,\left( \ref{Sikarov}\right) $, the highest
derivative for $\frac{\partial }{\partial z}$ is $\left( \frac{\partial }{%
\partial z}\right) ^{k}$, and is paired with the coefficient $\zeta _{k}$ in 
$\left( \ref{modify-terms}\right) $.

When \ $\zeta _{k}\neq 0,$ we take $h\left( z,\overline{z}\right) =\frac{%
\zeta _{k}}{\left\vert \zeta _{k}\right\vert ^{2}}\frac{1}{k!}$ $\left(
z-z_{0}\right) ^{k}w$, then $h$ solves $\left( \ref{modify-equ}\right) $.
This is because of the following: $h$ is holomorphic nearby $z_{0}$, so we
can ignore all terms in $Q\left( \frac{\partial }{\partial z},\frac{\partial 
}{\partial \overline{z}}\right) $ involving $\frac{\partial }{\partial 
\overline{z}}$; For the remaining terms in $Q\left( \frac{\partial }{%
\partial z},\frac{\partial }{\partial \overline{z}}\right) $, they must be
of the form $\left( \frac{\partial }{\partial z}\right) ^{l}$ with $0\leq
l\leq k$, and only $\left( \frac{\partial }{\partial z}\right) ^{k}h\left(
z_{0}\right) \neq 0$. \bigskip 

With this $h$, we reduce the cokernal equation to $\left\langle D_{u}%
\overline{\partial }_{j,J}\xi ,\eta \right\rangle =0$. Since $\eta $ is a
weak solution of $\left( D_{u}\overline{\partial }_{j,J}\right) ^{\ast }\eta
=0$ on $\Sigma $, by ellipticity of the $\left( D_{u}\overline{\partial }%
_{j,J}\right) ^{\ast }$ operator, the distribution solution $\eta $ is
smooth on $\Sigma $ (See \cite{Ho}). Since $\eta =0$ on $\Sigma \backslash
\left\{ z_{0}\right\} $ , $\eta =0$ on $\Sigma $. Then it is not hard to
conclude $\zeta =0$ by Cauchy integral formula as in \cite{OZ} and \cite{Oh}%
. Therefore the system of equations $\left( \ref{Lin-dbar}\right) $ and $%
\left( \ref{Lin-j}\right) $ is solvable for any $\eta \in W^{k,p}$ and $%
\alpha \in J_{hol}^{k}\left( T_{z_{0}}\Sigma ,T_{u\left( z_{0}\right)
}M\right) $.

There is one case left: that is when $\zeta _{k}=0$. We still need to show $%
\left( \eta ,\zeta \right) =\left( 0,0\right) $. If $k=1$, then $\zeta
_{1}=0\Leftrightarrow \zeta =0$ so it has been done as above. If $k>1$, we
notice that the cokernal equation $\left( \ref{coker-equ}\right) $ now is
the cokernal equation for the section $D\Upsilon _{k-1}$, since the $k$-th
jet is paired with $\zeta _{k}$ there, 
\begin{equation*}
\left( \left( D_{J,\left( j,u\right) }j_{hol}^{k}\right) \xi \left(
z_{0}\right) ,\zeta \right) _{z_{0}}=\left( \left( D_{J,\left( j,u\right)
}j_{hol}^{k-1}\right) \xi \left( z_{0}\right) ,\zeta \right) _{z_{0}}.
\end{equation*}
By induction assumption on $k$, $D\Upsilon _{k-1}$ has trivial cokernal
hence $\left( \eta ,\zeta \right) =\left( 0,0\right) $.

Last we raise the regularity from $W^{k+1,p}$ to $W^{N,p}$, for any $N>k$.
For $\eta \in W^{N-1,p}\subset W^{k,p}$, by the above argument we can find a
solution $\xi \in W^{k+1,p}$ in $\left( \ref{Lin-dbar}\right) $. By elliptic
regularity, the solution $\xi \in W^{N,p}$. Therefore $\left( \ref{Lin-dbar}%
\right) $ and $\left( \ref{Lin-j}\right) $ is solvable in $W^{N,p}$ setting.
This finishes induction hence the proof of Theorem.

\begin{remark}
In the above proof, the induction starts from $k=1$. In \cite{OZ}, $k=1$
case was treated in the framework of $1$-jet transversality at $\left(
u,z_{0}\right) $ where $du\left( z_{0}\right) =0$. The above proof includes
the $k=1$ case as well, but the way of choosing $h$ does not rely on $%
du\left( z_{0}\right) =0$ and applies to any $z_{0}$ on $\Sigma $.
\end{remark}

\begin{remark}
\label{fail-usual-jet}It is crucial that we use the holomorphic $k$-jet
bundle instead of the usual $k$-jet bundle to get the sujective property of $%
D\Upsilon _{k}$. Otherwise, as the usual jet evaluation involves mixed
derivatives, given $\zeta _{\left( k,0\right) }=0$ we can not reduce the
cokernal equation to the $\left( k-1\right) $ case by induction, and when $%
k=1$, $\zeta _{\left( 1,0\right) }=0$ does not imply $\zeta =0$. In the case 
$k=1,$ we can explicitly see why this submersion property fails in the usual 
$1$-jet bundle: for a $J$-holomorphic curve $u$ with $du\left( z_{0}\right)
=0$, and $\Gamma _{1}=\left( \overline{\partial },j^{k}\right) :\mathcal{F}%
_{1}\left( \Sigma ,M;\beta \right) \times \mathcal{J}_{\omega }\rightarrow 
\mathcal{H}^{^{\prime \prime }}\times J^{1}\left( \Sigma ,M\right) $,
calculations in \cite{OZ} yield 
\begin{equation*}
D\Gamma _{1}\left( \xi ,B\right) =\left( D_{u}\overline{\partial }_{j,J}\xi ;%
\text{ }D_{u}\overline{\partial }_{j,J}\xi \left( z_{0}\right)
,D_{u}\partial _{j,J}\xi \left( z_{0}\right) \right)
\end{equation*}%
therefore there is no solution for $\left( \eta ,\alpha _{\left( 0,1\right)
,}\alpha _{\left( 1,0\right) }\right) $ if $\eta \left( z_{0}\right) \neq
\alpha _{\left( 0,1\right) }$.
\end{remark}

However, if $du\left( z_{0}\right) \neq 0$ then the sujective property still
holds in the usual jet bundles. More precisely we have the following

\begin{theorem}
\label{usual-jet}At any $J$-holomrophic curve $\left( \left( u,j\right)
,z_{0},J\right) \in $ $\mathcal{M}_{1}^{\ast }\left( \Sigma ,M;\beta \right)
\subset \mathcal{F}_{1}\left( \Sigma ,M;\beta \right) \times \mathcal{J}%
_{\omega }$ with $du\left( z_{0}\right) \neq 0$, the linearization $D\Gamma
_{k}$ of the section 
\begin{equation*}
\Gamma _{k}=\left( \overline{\partial },\text{ }j^{k}\right) :\mathcal{F}%
_{1}\left( \Sigma ,M;\beta \right) \times \mathcal{J}_{\omega }\rightarrow 
\mathcal{H}^{^{\prime \prime }}\times J^{k}\left( \Sigma ,M\right)
\end{equation*}%
is a surjective. Especially the linearization $Dj^{k}$ of $k$-jet evaluation 
\begin{equation*}
j^{k}:\mathcal{F}_{1}\left( \Sigma ,M;\beta \right) \times \mathcal{J}%
_{\omega }\rightarrow J^{k}\left( \Sigma ,M\right)
\end{equation*}%
at $\left( \left( u,j\right) ,z_{0},J\right) $\ is surjective.
\end{theorem}

\begin{proof}
It is enough to show that the cokernal equation 
\begin{equation*}
\left\langle D_{u}\overline{\partial }_{j,J}\xi +\frac{1}{2}B\circ du\circ
j,\eta \right\rangle +\left( \left( D_{J,\left( j,u\right) }j^{k}\right) \xi
\left( z_{0}\right) ,\zeta \right) _{z_{0}}=0\text{, \ \ for all }\xi ,B%
\text{ \ \ \ \ \ \ \ }
\end{equation*}
only has trivial solution $\left( \eta ,\zeta \right) =\left( 0,0\right) $.
To do this, using standard argument in \cite{MS} we again get supp$\eta
\subset \left\{ z_{0}\right\} $. Given $\zeta \in J^{k}\left(
T_{z_{0}}\Sigma ,T_{u\left( z_{0}\right) }M\right) $, by Taylor polynomial
we can construct a smooth $\xi $ supported in arbitrarily small neighborhood
of $z_{0}\in \Sigma $, such that $\left( D_{J,\left( j,u\right)
}j^{k}\right) \xi \left( z_{0}\right) =\zeta $. When $du\left( z_{0}\right)
\neq 0$, by linear algebra (namely the abundance of $T_{J}\mathcal{J}%
_{\omega }$) and perturbation method we can construct $B\in T_{J}\mathcal{J}%
_{\omega }$ such that $D_{u}\overline{\partial }_{j,J}\xi +\frac{1}{2}B\circ
du\circ j=0$ on $\Sigma $ (see \cite{MS}). So we get from the cokernal
equation that $0+\left\vert \zeta \right\vert ^{2}=0$, i.e. $\zeta =0$. Let $%
B=0$ in the cokernal equation, we get $\left\langle D_{u}\overline{\partial }%
_{j,J}\xi ,\eta \right\rangle =0$ for all $\xi $. Then by elliptic
regularity we conclude $\eta =0$ on the whole $\Sigma $.
\end{proof}

The following theorem is a direct consequence of Theorem \ref{Submersion} by
applying Sard-Smale theorem.

\begin{theorem}
\label{S}Let $S$ be any smooth section of the holomorphic $k$-jet bundle $%
J_{hol}^{k}\left( \Sigma ,M\right) \rightarrow \mathcal{F}_{1}\left( \Sigma
,M;\beta \right) \times \mathcal{J}_{\omega }$. Then the section $\Upsilon
_{k}$ is transversal to the section $\left( 0,S\right) $. The moduli space 
\begin{equation*}
\mathcal{M}^{S}:=\left( j_{hol}^{k}\right) ^{-1}\left( S\right) \cap 
\mathcal{M}_{1}^{\ast }\left( \Sigma ,M;\beta \right) =\Upsilon
_{k}^{-1}\left( 0,S\right)
\end{equation*}%
is a Banach submanifold of codimension $2kn$ in $\mathcal{M}_{1}^{\ast
}\left( \Sigma ,M;\beta \right) $. Under the natural projection $\pi :%
\mathcal{F}_{1}\left( \Sigma ,M;\beta \right) \times \mathcal{J}_{\omega
}\rightarrow \mathcal{J}_{\omega }$, $\ $there exists $\mathcal{J}_{reg}$ $%
\subset \mathcal{J}_{\omega }$ of second category, such that for any $J\in 
\mathcal{J}_{reg}$, the modulis space $\mathcal{M}_{J}^{S}:=\mathcal{M}%
^{S}\cap \pi ^{-1}\left( J\right) $ is a smooth manifold in $\mathcal{M}%
_{1}^{\ast }\left( \Sigma ,M,J;\beta \right) ,$ with dimension 
\begin{equation*}
\dim \mathcal{M}_{J}^{S}=\dim \mathcal{M}_{1}^{\ast }\left( \Sigma
,M,J;\beta \right) -2kn,
\end{equation*}%
and all the elements in $\mathcal{M}_{J}^{S}$ are Fredholm regular.

\begin{remark}
In \cite{Oh}, the $S$ is the zero section of the holomorphic $k$-jet bundle $%
J_{hol}^{k}\left( \Sigma ,M\right) $, so $\mathcal{M}_{J}^{S}$ is the set of 
$J$-holomorphic curves with prescribed ramification degrees at the marked
points. The $J$-holomorphic curves in our moduli space $\mathcal{M}_{J}^{S}$
can obey more general constraint $S$. Similar to \cite{Oh}, the theorem also
has the version with more than one marked point. Also the constaint $S$ need
not to be a full section over the base, but only a closed submanifold in $%
J_{hol}^{k}\left( \Sigma ,M\right) $ whose tangent space projects onto the
horizontal distribution of the bundle $J_{hol}^{k}\left( \Sigma ,M\right)
\rightarrow $ $\mathcal{F}_{1}\left( \Sigma ,M;\beta \right) \times \mathcal{%
J}_{\omega }$, because the essential part in the proof the theorem is that $%
D\Upsilon _{k}|_{\left( 0,S\right) }$ is surjective.
\end{remark}
\end{theorem}

The theorem appears to be a good start of studying moduli spaces of $J$%
-holomorphic curves satisfying general jet constraints in the holomorphic
jet bundle; for example, moduli spaces of $J$-holomorphic curves with self
tangency. Also in \cite{CM}, jet constraints from symplectic hypersurfaces
were used to get rid of multicovering bubbling spheres. This enables them to
define genus zero Gromov-Witten invariants without abstract perturbations.

The above theorem tells that the moduli spaces $\mathcal{M}_{J}^{S}$ are
well-behaved, and the $\left\{ J_{t}\right\} _{0\leq t\leq 1}$ family
version of the above theorem tells that they are cobordant to each other by
moduli spaces $\left\{ \mathcal{M}_{J_{t}}^{S}\right\} _{0\leq t\leq 1}$ for
generic path $J_{t}\subset \mathcal{J}_{\omega }$. It is interesting to see
if the moduli spaces $\mathcal{M}_{J}^{S}$ can be used to construct new
symplectic invariants.

\end{document}